\documentclass{amsart}
\usepackage[all]{xy}
\usepackage{mathtools}
\usepackage{hyperref}
\usepackage{cleveref}

\newtheorem*{theorem*}{Theorem}
\newtheorem*{theoremA*}{Theorem A}
\newtheorem*{theoremB*}{Theorem B}
\newtheorem{theorem}{Theorem}[section]
\newtheorem{lemma}[theorem]{Lemma}
\newtheorem{proposition}[theorem]{Proposition}
\newtheorem{conjecture}[theorem]{Conjecture}

\theoremstyle{definition}
\newtheorem{definition}[theorem]{Definition}
\newtheorem{notations}[theorem]{Notations}
\newtheorem{assumption}[theorem]{Assumption}
\newtheorem{example}[theorem]{Example}

\newtheorem{caution}[theorem]{Caution}

\theoremstyle{remark}
\newtheorem{remark}[theorem]{Remark}

\newcommand{\ord}{\textup{ord}}
\newcommand{\eord}{\textup{eord}}
\newcommand{\Eord}{\textup{Eord}}
\newcommand{\Mon}{\textup{Mon}}
\newcommand{\gr}{\textup{gr}}

\numberwithin{equation}{section}

\begin{document}


\title{Strongly Lech-independent ideals and Lech's conjecture}
\author{Cheng Meng}
\address{Cheng Meng\\ Yau Mathematical Sciences Center, Tsinghua University, Beijing 100084, China. \emph{Email:} {\rm cheng319000@tsinghua.edu.cn}}
\date{\today}
\maketitle

\begin{abstract}
We introduce the notion of strongly Lech-independent ideals as a generalization of Lech-independent ideals defined by Lech and Hanes, and use this notion to derive inequalities on multiplicities of ideals. In particular, we prove that if $(R,\mathfrak{m}) \to (S,\mathfrak{n})$ is a flat local extension of local rings with $\dim R=\dim S$, the completion of $S$ is the completion of a standard graded ring over a field $k$ with respect to the homogeneous maximal ideal, and the completion of $\mathfrak{m}S$ is the completion of a homogeneous ideal, then $e(R) \leq e(S)$.
\end{abstract}

\section{introduction}\label{se:intro}
Around 1960, Lech made the following remarkable conjecture on the Hilbert-Samuel multiplicities in \cite{lech1960note}:
\begin{conjecture}\label{conj: Lecj}Let $(R,\mathfrak{m}) \to (S,\mathfrak{n})$ be a flat local extension of local rings. Then $e(R) \leq e(S)$.
\end{conjecture}
As the Hilbert-Samuel multiplicity measures the singularity of a ring, this conjecture roughly means that the singularity of $R$ is no worse than that of $S$ if $(R,\mathfrak{m}) \to (S,\mathfrak{n})$ is a flat local extension. This conjecture has now stood for more than sixty years and remains open in most cases. It has been proved in the following cases:
\begin{enumerate}
\item dim$R \leq$ 2 \cite{lech1960note};

\item $S/\mathfrak{m}S$ is a complete intersection \cite{lech1960note};

\item $R$ is a strict complete intersection \cite{herzog1991linear};

\item dim$R$ = 3 and $R$ has equal characteristic \cite{ma2017lech};

\item $R$ is a standard graded ring over a perfect field (localized at the homogeneous maximal ideal) \cite{ma2020lim}.
\end{enumerate}
For other results see \cite{hanes2001length}, \cite{hanes1999special}, \cite{herzog1990lech} and \cite{ma2014frobenius}.
In this paper the key concept is a new notion called \emph{strong Lech-independence}, which is a natural generalization of Lech-independence introduced in \cite{lech1964inequalities} and explored in \cite{hanes2001length}. By definition, an ideal $I \subset S$ is strongly Lech-independent if for any $i$, $I^i/I^{i+1}$ is free over $S/I$, and a sequence of elements is strongly Lech-independent if it forms a minimal generating set of a strongly Lech-independent ideal. There are two typical examples of strongly Lech-independent ideals: ideals generated by a regular sequence, and the maximal ideal of a local ring.

Under strong Lech-independence assumption, we can calculate the colength of powers of an ideal using the data on the monomials of a minimal generating set of the ideal, thus we can derive inequalities on multiplicities. The main result on multiplicities of ideals is the following particular case of Lech's conjecture:

\begin{theoremA*}[See \Cref{thm: maintheoremA}]Let $(R, \mathfrak{m}) \to (S, \mathfrak{n})$ be a flat local map of Noetherian local rings. Assume:
\begin{enumerate}
\item Up to completion $S$ is standard graded;
\item $\dim R=\dim S$;
\item Up to completion $\mathfrak{m}S$ is homogeneous with homogeneous generators of degrees $t_1\leq t_2 \leq \ldots \leq t_r$.
\end{enumerate}
Then $e(S) \geq e(R)t_1...t_{r-d}$.  
\end{theoremA*}
This theorem will lead to the inequality $e(S) \geq e(R)$ because we always have $r \geq d$ and $t_1 \geq 1$.

We can also derive an inequality of the other direction, that is, we can find an upper bound of $e(S)$ using strong Lech-independence condition. For $f \in S$ where $(S,\mathfrak{n})$ is a Noetherian local ring, let $\ord(f)$ be the $\mathfrak{n}$-adic order of $f$, see \Cref{def: nadicorder}. Let $\bar{v}(x)=\lim_{n \to \infty}\ord(x^n)/n$, then $\bar{v}$ is a well-defined function from $S$ to $\mathbb{R}\cup\{\infty\}$ called the asymptotic Samuel function. Its value at any element of $S$ is either a nonnegative rational number or $\infty$. Then we have the following upper bound of $e(S)$:
\begin{theoremB*}[See \Cref{thm: maintheoremB}]
Assume $S$ has equal characteristic, $I$ is an $S$-ideal which is strongly Lech-independent and $l(S/I)<\infty$. Let $d=\dim S$. Assume $I$ is minimally generated by $(x_1,...,x_r)$, $\bar{v}(x_i)=q_i$ where $q_1 \leq q_2 \leq ... \leq q_r$. Then $e(S) \leq e(I)/q_1...q_{d-1}q_d$.
\end{theoremB*}

The paper is organized in the following way. In \Cref{se:notation} we introduce some notations for this paper. In \Cref{se:standardset}, we start with the definition of a standard set, along with some basic definitions and properties on the set of monomials in a polynomial ring. In \Cref{se:SLI definition}, we define strong Lech-independence and expansion properties and prove some equivalent conditions. We will prove that strong Lech-independence implies certain expansion property. In \Cref{se:Maintheorem}, we use strong Lech-independence to analyze the colengths of powers of ideals and derive inequalities on multiplicities.

\section{Preliminaries, notations and assumptions}\label{se:notation}
In this section, we set up some notations and assumptions that will be used throughout this paper. We will also mention some preliminary results.

In this paper, we make the following assumptions:
\begin{assumption}\label{assumption2.1}
We assume $(S,\mathfrak{n})$ is a Noetherian local ring with maximal ideal $\mathfrak{n}$. We assume $I$ is an $S$-ideal with minimal generators $x_1,x_2,\ldots,x_r$, and $r$ is the number of minimal generators of $I$.    
\end{assumption}

Apart from \Cref{assumption2.1}, we also use the following notations. The symbol $I^i,i \in \mathbb{N}_+$ denotes the $i$-th ordinary power of $I$, $I^0$ denotes the unit ideal $S$, and $I^\infty$ denotes the zero ideal $(0)$. We assume $k$ is a field, $P$ is a polynomial ring over $k$. The variables of $P$ are $T_1,T_2,\ldots,T_n$. The indeterminates of a multigraded Hilbert series are denoted by $z_1,z_2,\ldots$, and the indeterminate of a $\mathbb{Z}$-graded or $\mathbb{N}$-graded Hilbert series is denoted by $z$. We will use lower case letters for sequences of elements except for the sequence $\underline{z}=(z_1,z_2,\ldots)$ which is used in the notation of multigraded Hilbert series. In particular, we write $x=(x_1,x_2,\ldots,x_r)$, which is a minimal generating sequence of $I$. The cardinality of a finite set $A$ is denoted by $|A|$.

We will assume $k$ is a general field in \Cref{se:standardset} and $k=S/\mathfrak{n}$ is the residue field of $S$ in \Cref{se:SLI definition} and \Cref{se:Maintheorem}. Therefore, the results in \Cref{se:standardset} for a general $k$ will apply to the corresponding results in \Cref{se:SLI definition} and \Cref{se:Maintheorem}. We use the symbol $\dim$ for the Krull dimension of rings and modules, and $\dim_k$ for the vector space dimension over $k$. The symbol $l(N)$ represents the length of an $S$-module $N$. If $S$ has a coefficient field $k$ and $N$ has finite length, then $l(N)=\dim_k(N)<\infty$.

For a monomial $u$ in the polynomial ring $P=k[T_1,\ldots,T_n]$, we can express $u$ in terms of products of variables, that is, $u=T_1^{a_1}T_2^{a_2}\ldots T_n^{a_n}$. To avoid expressing $u$ explicitly, we will define $\deg_{T_i}u=a_i$ for $1 \leq i \leq n$. For a sequence $y=(y_1,\ldots,y_n)$ of $S$-elements whose length is equal to the number of variables in $P$, the expression $y_1^{a_1}y_2^{a_2}\ldots y_n^{a_n}=y_1^{\deg_{T_1}u}y_2^{\deg_{T_2}u}\ldots y_n^{\deg_{T_n}u} \in S$ is the evaluation of $u$ at $y=(y_1,\ldots,y_n)$, denoted by $u(y)$. The most commonly used example of this notation is the case when $n=r$ and $y=x=(x_1,\ldots,x_r)$. In this case $u(x)=x_1^{\deg_{T_1}u}x_2^{\deg_{T_2}u}\ldots x_r^{\deg_{T_r}u}$. If $n=r$ and $y=(z^{t_1},z^{t_2},\ldots,z^{t_r})$, then $u(y)=z^{t_1a_1+t_2a_2+\ldots+t_ra_r}$ where $a_i=\deg_{T_i}u$.

We identify $\mathbb{R}[z]_{(z)}$ as a subring of $\mathbb{R}[[z]]$ through the natural embedding $\mathbb{R}[z]_{(z)} \subset \mathbb{R}[[z]]$. For a power series $a(z) \in \mathbb{R}[[z]]$ that falls into $\mathbb{R}[z]_{(z)}$, it is a rational function where $z=0$ is not a pole, and in this case we also view it as a meromorphic function on $\mathbb{C}$ and ignore its radius of convergence. With this understanding, we have:
\begin{lemma}\label{lem: residueleadingterm}Let $d$ be a positive integer, $a(z) = \sum_{i \geq 0} a_iz^i \in \mathbb{R}[z]_{(z)}$ be a rational series satisfying the following properties:
\item (P1) $a(z)$ only has poles at roots of unity;

\item (P2$_d$) $z = 1$ is a pole of $a(z)$ with order $d$;

\item (P3$_d$) The orders of poles of $a(z)$ except for $1$ are less than $d$.

Then we have
\begin{align}\label{equation2.1}
\lim_{z \to 1}(\sum_{i \geq 0} a_iz^i)(1-z)^d=\lim_{k \to \infty}\frac{(d-1)!}{(d+k-1)!}\frac{\partial^{k}a(0)}{\partial z^{k}}\\
=\lim_{k \to \infty}\frac{(d-1)!k!}{(d+k-1)!}a_k
=\lim_{k \to \infty}\frac{(d-1)!}{k^{d-1}}a_k.    
\end{align}
The first limit is understood by viewing $\sum_{i \geq 0} a_iz^i$ as a meromorphic function on $\mathbb{C}$. In particular, if $a(z)$ has a single pole at $z=1$ of order $d$, then $a_k=Ck^{d-1}+o(k^{d-1})$ for some $C \neq 0$.
\end{lemma}
\begin{proof}We prove \Cref{equation2.1} first. We can express $a(z)$ using partial-fraction decomposition; the existence of a partial-fraction decomposition of a rational function can be seen in, for example, Section 2.1.4 of \cite{martin1966complex}. To be precise, let $U$ be the set of poles of $a(z)$, then there exist finitely many complex numbers $e_{i,\xi}, 1 \leq i \leq d-1, \xi \in U$, a complex number $e_0 \neq 0$, and a polynomial $b(z)$ such that
\begin{equation}\label{equation2.3}
a(z)=\sum_{1 \leq i \leq d-1, \xi \in U}e_{i,\xi}(\xi-z)^{i-d}+e_0(1-z)^{-d}+b(z).
\end{equation}
Let $L$ be the map $a(z) \to \lim_{k \to \infty}\frac{(d-1)!}{(d+k-1)!}\frac{\partial^{k}a(0)}{\partial z^{k}}$. Then it is $\mathbb{Q}$-linear when it is well-defined. We apply $L$ to each term in the right side of \Cref{equation2.3}. If $1 \leq i \leq d-1$,
$$L((\xi-z)^{i-d}) = \lim_{k \to \infty} \frac{(d-1)!(d-i+k-1)!}{(d-i-1)!(d+k-1)!}(\xi-0)^{i-d-k} = 0$$
as $(\xi-0)^{i-d-k}$ is bounded and $\frac{(d-1)!(d-i+k-1)!}{(d-i-1)!(d+k-1)!}$ goes to $0$,
$$L((1-z)^{-d}) = \lim_{k \to \infty} \frac{(d-1)!(d+k-1)!}{(d-1)!(d+k-1)!}(1-0)^{i-d-k} = 1,$$
and $L(b(z)) = 0$ as $b(z)$ is a polynomial and $d>0$. This means the right side of \Cref{equation2.1} is $L(a(z)) = e_0$. The left side of \Cref{equation2.1} is also $e_0$, so they are equal. So the first equality is proved. In particular $e_0$ is a real number. The second equality is true since $\frac{\partial^{k}a(0)}{\partial z^{k}}=k!a_k$ and the third equality is true since for fixed $d \geq 1$, $\lim_{k \to \infty}\frac{k!k^{d-1}}{(d+k-1)!}=\lim_{k \to \infty}\Pi_{1 \leq i \leq d-1}(1+\frac{i}{k})=1$. Finally if $a(z)$ has a single pole at $z=1$ of order $d$, then it satisfies (P1), (P2$_d$), and (P3$_d$), so $\lim_{k \to \infty}\frac{(d-1)!}{k^{d-1}}a_k$ exists.
\end{proof}
If $l(S/I)<\infty$, the Hilbert Samuel multiplicity of $I$ is
$$\lim_{t \to \infty}\frac{(d-1)!l(I^t/I^{t+1})}{t^{d-1}}=\lim_{t \to \infty}\frac{d!l(S/I^t)}{t^d}=e(I).$$
We say $e(\mathfrak{n})=e(S)$ is the multiplicity of $S$.

\section{Standard sets in a polynomial ring}\label{se:standardset}
Let $k$ be a field, $n$ be a positive integer, $P=k[x_1,\ldots,x_n]$. In this section, we will define the concept of a standard set which is a certain kind of subset of $P$. This concept will be used in \Cref{se:SLI definition} and \Cref{se:Maintheorem}.

\begin{definition}\label{def: standardset}An ideal $J$ of $P$ is called a \emph{monomial ideal}, if $J$ is generated by monomials. A set of monomials $\Gamma$ is called a \emph{standard set of monomials}, or a standard set for short, if $\Gamma$ is a subset of monomials in $P$ such that if $u$ is in $\Gamma$, then every monomial dividing $u$ is in $\Gamma$.
\end{definition}
Let $\Mon(\cdot)$ be the set of all the monomials in a polynomial ring or a monomial ideal. For a standard set $\Gamma$, let $\Gamma_i$ be the monomials of degree $i$ in $\Gamma$. A standard set is closed under taking factors, hence its complement is closed under taking multiples, which means that the complement is just the set of all monomials in a monomial ideal. Hence we have:
\begin{proposition}\label{prop: standardset=mon ideal}$\Gamma$ is a standard set if and only if for some monomial ideal $I_{\Gamma}$, $\Mon(P)\backslash \Gamma=\Mon(I_{\Gamma})$. This builds a bijection between the set of standard sets and the set of monomial ideals in $P$.
\end{proposition}

Next we derive some data of the multigraded ring $P/I_\Gamma$ from $\Gamma$ where $\Gamma$ is a standard set. Here the multigraded Hilbert function and multigraded Hilbert series are defined in \cite{SturmfelsCA}, Definition 1.10 and Definition 8.14; the reader can refer to Chapters 1 and 8 in \cite{SturmfelsCA} for relevant knowledge on multigraded Hilbert series.
\begin{proposition}\label{prop: standardset Hilbertseries}Let $\Gamma \subset \Mon(P)$ be a standard subset. 
\begin{enumerate}
\item The multigraded Hilbert series of $P/I_\Gamma$ is $HS_{P/I_\Gamma}(\underline{z})=\sum_{u \in \Gamma} u(\underline{z}).$
\item The Hilbert series of $P/I_\Gamma$ is $HS_{P/I_\Gamma}(z)=HS_{P/I_\Gamma}(z,z,...,z)$.
\item The Krull dimension $d$ of $P/I_\Gamma$ is the order of $HS_{P/I_\Gamma}(z)$ at the pole $z=1$; the multiplicity of $P/I_\Gamma$ is $\lim_{z \to 1}HS_{P/I_\Gamma}(z)(1-z)^d$.
\end{enumerate}
\end{proposition}
\begin{proof}
(1) just comes from the definition of multigraded Hilbert series since for any multi-index $a=(a_1,\ldots,a_n)$, $\dim_k (P/I_\Gamma)_a=1$ if $T_1^{a_1}\ldots T_n^{a_n} \in \Gamma$ and $\dim_k (P/I_\Gamma)_a=0$ otherwise. (2) comes from Definition 1.10 of \cite{SturmfelsCA}. The first part of (3) is a consequence of the equality $\dim M=d(M)=\delta(M)$; for the meaning of these symbols and the proof see Theorem 13.4 of \cite{Matsumuracommutativering}. Here $d(M)$ is the degree of the Hilbert-Samuel polynomial, and it is equal to the order of poles of the corresponding Hilbert series at $z=1$ by \Cref{lem: residueleadingterm}. The second part is true by \Cref{lem: residueleadingterm}.   
\end{proof}
Sometimes we only care about the standard set $\Gamma$, not the monomial ideal $I_{\Gamma}$. So we make the following convention.
\begin{definition}\label{def: standardsetHilbert}Let $\Gamma$ be a standard set in a polynomial ring $P$. We define the multigraded Hilbert series, the Hilbert series, dimension and multiplicity of $\Gamma$ to be that of $P/I_\Gamma$.
\end{definition}
In general, $\Gamma$ is an infinite set, but there is a way to express it in terms of finitely many monomials and finitely many polynomial subrings.
\begin{proposition}[\cite{vasconcelos2004computational}, Stanley decomposition]\label{prop: Stanleydecomp} For each standard set $\Gamma$, there exists a finite set of pairs $(u_i,S_i)_{i \in \Lambda}$ where every $u_i$ is a monomial in $\Gamma$ and every $S_i$  is a subset of variables such that $P/I_{\Gamma}=\oplus_{i \in \Lambda} u_ik[S_i]$ as a $k$-vector space. In this case, $\Gamma$ is the disjoint union of $u_i \cdot \Mon(k[S_i])$ where $i \in \Lambda$.
\end{proposition}
We call such a partition of $\Gamma$ a \emph{Stanley decomposition} of $\Gamma$, denoted by $(u_i,S_i)_{i \in \Lambda}$. we also have the following proposition of the Stanley decomposition.
\begin{proposition}[\cite{vasconcelos2004computational}]\label{prop: Stanleyconsequence}Let $\Gamma$ be a standard set with Stanley decomposition $(u_i,S_i)_{i \in \Lambda}$. Then the multigraded Hilbert series of $\Gamma$ is $\sum_{i \in \Lambda} \frac{u_i(\underline{z})}{\Pi_{T_j \in S_i}(1-z_j)}$. The dimension $d$ of $\Gamma$ is $\max{|S_i|}$. The multiplicity of $\Gamma$ is the number of $i$ such that $|S_i|=d$.
\end{proposition}

\section{strong Lech-independence and expansion property}\label{se:SLI definition}
In this section, we assume $k$ is the residue field of $S$, $P=k[T_1,...,T_r]$ is the polynomial ring over $k$ in exactly $r$ variables where $r$ is the number of minimal generators of $I$. We will define a property on the ideal $I$ called strong Lech-independence which is the key concept in the proofs in \Cref{se:Maintheorem} that allows us to deduce inequalities on multiplicities of ideals. This concept is a generalization of Lech-independence defined in \cite{lech1964inequalities}:
\begin{definition}\label{def: SLI}We say that $I$ is \emph{Lech-independent} if $I/I^2$ is free over $S/I$. We say that $I$ is \emph{strongly Lech-independent} if $I^i/I^{i+1}$ is free over $S/I$ for any $i \geq 1$. We say that a sequence of elements $x_1,...,x_r$ is Lech-independent (resp. strongly Lech-independent), if it forms a minimal generating set of an ideal which is Lech-independent (resp. strongly Lech-independent).
\end{definition}
We would like to point out that the concept of permissible ideal defined in page 33 in \cite{permissibleref} is equivalent to $I$ being a strongly Lech-independent prime ideal and $S/I$ being regular.

We have the following equivalent conditions for strong Lech-independence.
\begin{proposition}\label{prop: SLI=grISfree}The following are equivalent for $I$.
\begin{enumerate}
\item $I$ is strongly Lech-independent.

\item $\gr_I(S)$ is free over $S/I$.

\item $\gr_I(S)$ is flat over $S/I$.
\end{enumerate}
\end{proposition}
\begin{proof}It suffices to prove $(3) \Rightarrow (1)$. If $\gr_I(S)$ is flat over $S/I$, then for any $i \geq 1$, $I^i/I^{i+1}$ is flat over $S/I$ because it is a direct summand of $\gr_I(S)$. But it is finitely generated over the local ring $S/I$, so it is free. So $I$ is strongly Lech-independent by definition.
\end{proof}
Next we introduce an expansion property for a sequence of elements of the ring $S$ and relate it to strong Lech-independence.
\begin{definition}\label{def: lifting}We say a set-theoretic map $\sigma:S/I \to S$ is a \emph{lifting which preserves 0}, or a lifting for short, if $\sigma(0)=0$ and the composition of $\sigma$ with the natural quotient map $\pi: S \to S/I$ is the identity map on $S/I$.
\end{definition}
Roughly speaking, $\sigma$ picks a representative for each coset in $S/I$ and picks $0$ element as the representative of $0$ coset. Such liftings always exist by axiom of choice. Here we present an example to produce such a lifting in finite steps.
\begin{example}\label{eg: klinearlifting}
Suppose $S$ has a coefficient field $k$, which implies that $S/I$ is a $k$-vector space. Choose a $k$-basis of $S/I$, then they are of the form $\{f_i+I\},i \in \Lambda$ where $f_i \in S$ for any $i$. We fix a choice of such $f_i$. Expanding the set-theoretic map $f_i+I \to f_i$ $k$-linearly gives a map $\sigma: S/I \to S$. Then $\sigma$ is a well-defined lifting.    
\end{example}
\begin{proof}
The well-definedness comes from the fact that any set-theoretic map from a basis of a $k$-vector space to another $k$-vector space extends to a $k$-linear map between the two spaces. As a $k$-linear map, $\sigma$ maps $0$ to $0$. Also, $\pi$ is a $k$-linear map, so $\pi\sigma$ is a $k$-linear map. By definition, the restriction of $\pi\sigma$ on the basis element is the identity map, thus $\pi\sigma$ is the identity $k$-linear map on $S/I$.    
\end{proof}
\begin{definition}[Expansion property]\label{def: expansionprop}
We fix two nonnegative integers $i<j$ and a subset $\Gamma$ of $\Mon(P)$. Assume $I$ and $x=(x_1,\ldots,x_r)$ are as in \Cref{assumption2.1}. 
\begin{enumerate}
\item We say $(x_1,...,x_r)$ is \emph{$\Gamma$-expandable from degree $i$ to $j$}, if for any lifting $\sigma:S/I \to S$, every element $f \in I^i$ has a representation
$$f=\sum_{u \in \Gamma_m, i \leq m \leq j-1} f_{u}u(x)+g$$
such that for any $u$, $f_{u} \in \sigma(S/I)$ and $g \in I^j$, and the choice of $f_u$ and $g$ making the equality hold is unique.
\item If $S$ is complete, we say that $(x_1,...,x_r)$ is $\Gamma$-expandable from degree $i$ to $\infty$, if for any lifting $\sigma:S/I \to S$, every element $f \in I^i$ has a representation
$$f=\sum_{u \in \Gamma_m, i \leq m} f_{u}u(x)$$
such that for any $u$, $f_{u} \in \sigma(S/I)$, and the choice of $f_u$ making the equality hold is unique. 
\item We say that $(x_1,...,x_r)$ is $\Gamma$-expandable if it is expandable from degree 0 to $\infty$. 
\item The two expressions $f=\sum_{u \in \Gamma_m, i \leq m \leq j-1} f_{u}u(x)+g$ and $f=\sum_{u \in \Gamma_m, i \leq m} f_{u}u(x)$ in (1) and (2) are called the expansion of $f$ with respect to $\Gamma$ and the lifting $\sigma$, or simply the expansion of $f$ if $\Gamma$ and $\sigma$ are clear.
\item We say an ideal is $\Gamma$-expandable from degree $i$ to $j$ or $\infty$ if one minimal generating sequence of the ideal is $\Gamma$-expandable from degree $i$ to $j$ or $\infty$.
\end{enumerate}
\end{definition}
We stress that when $I$ is $\Gamma$-expandable from degree $i$ to $\infty$, the ambient ring $S$ is implicitly assumed to be complete, otherwise this notion does not make sense as the infinite sum may not converge to an element in $S$.

Strong Lech-independence can be described using the expansion property. We start with two lemmas:
\begin{lemma}\label{lem: expansioni1i2i3}Let $i_1,i_2$ be nonnegative integers, and $i_3$ is either a positive integer or $\infty$ such that $i_1 < i_2 < i_3$. Consider 3 conditions on a sequence $(x_1,...,x_r)$.
\begin{enumerate}
\item $(x_1,...,x_r)$ is $\Gamma$-expandable from degree $i_1$ to $i_2$

\item $(x_1,...,x_r)$ is $\Gamma$-expandable from degree $i_1$ to $i_3$

\item $(x_1,...,x_r)$ is $\Gamma$-expandable from degree $i_2$ to $i_3$
\end{enumerate}
Then two of them imply the third one.
\end{lemma}
\begin{proof}Let $I=(x_1,...,x_r)$.

Assume (1) and (2) are true, then for any $f \in I^{i_2} \subset I^{i_1}$, by (2) we have
$$f=\sum_{u \in \Gamma_m, i_1 \leq m \leq i_3-1} f_{u}u(x)+f_{i_3},f_{i_3} \in I^{i_3}.$$
Let
$$f'=\sum_{u \in \Gamma_m, i_1 \leq m \leq i_2-1} f_{u}u(x),$$
then $f'-f \in I^{i_2}$, so $f' \in I^{i_2}$. By (1) the unique expansion of $f'$ modulo $I^{i_2}$ is itself. So $f'=f_u=0$ for all $u \in \Gamma_m, i_1 \leq m \leq i_2-1$ and hence we have
$$f=\sum_{u \in \Gamma_m, i_2 \leq m \leq i_3-1} f_{u}u(x).$$
This shows the existence. The uniqueness just follows from (2) because an expansion from degree $i_2$ to $i_3$ can be viewed as an expansion from degree $i_1$ to $i_3$ by adding $0$ terms.

Assume (1) and (3) are true. Let $f \in I^{i_1}$, then by (1)
$$f=\sum_{u \in \Gamma_m, i_1 \leq m \leq i_2-1} f_{u}u(x)+g,$$
where $g \in I^{i_2}$. By (3),
$$g=\sum_{u \in \Gamma_m, i_2 \leq m \leq i_3-1} g_{u}u(x)+h,$$
where $h \in I^{i_3}$. Thus
$$f=\sum_{u \in \Gamma_m, i_1 \leq m \leq i_2-1} f_{u}u(x)+\sum_{u \in \Gamma_m, i_2 \leq m \leq i_3-1} g_{u}u(x)+h$$
is a representation of $f$. This shows the existence. For uniqueness, let
$$\sum_{u \in \Gamma_m, i_1 \leq m \leq i_2-1} f'_{u}u(x)+\sum_{u \in \Gamma_m, i_2 \leq m \leq i_3-1} g'_{u}u(x)+h'$$
be another representation of $f$ where $h' \in I^{i_3}$. Then 
$$f=\sum_{u \in \Gamma_m, i_1 \leq m \leq i_2-1} f_{u}u(x)+(\sum_{u \in \Gamma_m, i_2 \leq m \leq i_3-1} g_{u}u(x)+h)$$
$$=\sum_{u \in \Gamma_m, i_1 \leq m \leq i_2-1} f'_{u}u(x)+(\sum_{u \in \Gamma_m, i_2 \leq m \leq i_3-1} g'_{u}u(x)+h').$$
Both are expansions of $f$ from degree $i_1$ to $i_2$. Hence by (1), $f_u=f'_u$ for any $u \in \Gamma_m, i_1 \leq m \leq i_2-1$, and
$$\sum_{u \in \Gamma_m, i_2 \leq m \leq i_3-1} g_{u}u(x)+h=\sum_{u \in \Gamma_m, i_2 \leq m \leq i_3-1} g'_{u}u(x)+h'.$$
By (3) $g_u=g'_u$ and $h=h'$, which proves the uniqueness.

Assume (2) and (3) are true. Then for any $f \in I^{i_1}$, by (2)
$$f=\sum_{u \in \Gamma_m, i_1 \leq m \leq i_3-1} f_{u}u(x)+h,h \in I^{i_3}.$$
Then
$$f=\sum_{u \in \Gamma_m, i_1 \leq m \leq i_2-1} f_{u}u(x)+(\sum_{u \in \Gamma_m, i_2 \leq m \leq i_3-1} f_{u}u(x)+h),$$
so the representation exists. Suppose there is another expression $$f=\sum_{u \in \Gamma_m, i_1 \leq m \leq i_2-1} f'_{u}u(x)+g, g \in I^{i_2}.$$
Then by (3)
$$g=\sum_{u \in \Gamma_m, i_2 \leq m \leq i_3-1} g_{u}u(x)+h',h' \in I^{i_3}.$$
So
$$f=\sum_{u \in \Gamma_k, i_1 \leq m \leq i_2-1} f'_{u}u(x)+\sum_{u \in \Gamma_m, i_2 \leq m \leq i_3-1} g_{u}u(x)+h'.$$
Hence $f'_u=f_u$ for any $u \in \Gamma_m, i_1 \leq m \leq i_2-1$ by the uniqueness of (2) which implies
$$g=\sum_{u \in \Gamma_m, i_2 \leq m \leq i_3-1} g_{u}u(x)+h'=\sum_{u \in \Gamma_m, i_2 \leq m \leq i_3-1} f_{u}u(x)+h,$$
so the uniqueness of (1) is proved.
\end{proof}
\begin{lemma}\label{lem: expansionijinfty}Assume $S$ is complete. Let $i$ be an integer. Let $i_1<i_2<...$ be a sequence of integers going to infinity and assume that $i<i_1$. Suppose $(x_1,...,x_r)$ is $\Gamma$-expandable from degree $i$ to $i_j$ for any $j$. Then $(x_1,...,x_r)$ is $\Gamma$-expandable from degree $i$ to $\infty$.
\end{lemma}
\begin{proof}Let $I=(x_1,...,x_r)$ and take $f \in I^i$. Let
$$f=\sum_{u \in \Gamma_m, i \leq m \leq i_j-1} f_{j,u}u(x)+g_j, g_j \in I^{i_j}.$$
Suppose $j<j'$. Then
$$\sum_{u \in \Gamma_m, i \leq m \leq i_j-1} f_{j,u}u(x)+g_j=\sum_{u \in \Gamma_m, i \leq m \leq i_{j'}-1} f_{j',u}u(x)+g_{j'},$$
so $$\sum_{u \in \Gamma_m, i \leq m \leq i_j-1} f_{j,u}u(x)=\sum_{u \in \Gamma_m, i \leq m \leq i_j-1} f_{j',u}u(x)\textup{ modulo }I^{i_j}.$$
By the uniqueness of the representation from degree $i$ to degree $i_j$, $f_{j,u}=f_{j',u}$ for any $j, j', u \in \Gamma_m$ where $i \leq m \leq i_j-1$. So for any pair $(j,u)$ whenever $u \in \Gamma_m, i \leq m \leq i_j-1$, $f_{j,u}$ is independent of the choice of $j$, so we can denote it by $f_u$. The expression $\sum_{u \in \Gamma_m, i \leq m < \infty} f_uu(x)$ makes sense in $S$ because $S$ is complete. We have
$f-\sum_{u \in \Gamma_m, i \leq m < \infty} f_uu(x) \in I^{i_j}$ for any $j$, so it is 0. Therefore,
$$f=\sum_{u \in \Gamma_m, i \leq m < \infty} f_uu(x)$$
is a representation of $f$. The uniqueness can be proved modulo $I^{i_j}$ for any $j$.
\end{proof}
The previous two lemmas lead to the following proposition which characterizes strong Lech-independence.
\begin{proposition}\label{prop: SLI=expandable}The following are equivalent.
\begin{enumerate}
\item $I$ is strongly Lech-independent.

\item For every minimal generating sequence $x_1,...,x_r$ of $I$ there is a standard subset $\Gamma$ of $\Mon(P)$ such that $I^i/I^{i+1}$ is free over $S/I$ with basis $u(x)$, with $u \in \Gamma_i$.

\item For every minimal generating sequence $x_1,...,x_r$ of $I$ there is a standard subset $\Gamma$ of $\Mon(P)$ such that for any $i$, $x_1,...,x_r$ is $\Gamma$-expandable from degree $i$ to $i+1$.

\item For every minimal generating sequence $x_1,...,x_r$ of $I$ there is a standard subset $\Gamma$ of $\Mon(P)$ such that for any $i<j$, $x_1,...,x_r$ is $\Gamma$-expandable from degree $i$ to $j$.

\item When $S$ is complete, for every minimal generating sequence $x_1,...,x_r$ of $I$ there is a standard subset $\Gamma$ of $\Mon(P)$ such that for any $i$, $x_1,...,x_r$ is $\Gamma$-expandable from degree $i$ to $\infty$.
\end{enumerate}
\end{proposition}
\begin{remark}
In \Cref{prop: SLI=expandable}, (1) means $I^i/I^{i+1}$ is free over $S/I$ for any $i \geq 1$, and (2) means (1) and we can choose free bases coming from a single standard set for all degrees. So priorly, (2) is stronger than (1).    
\end{remark}
\begin{proof}
(1) implies (2): Let $I=(x_1,...,x_r)$. Since $I^i/I^{i+1}$ is free, the preimage of a $k$-basis of $I^i/I^{i+1} \otimes_S S/\mathfrak{n}$ forms an $S/I$-basis of $I^i/I^{i+1}$. Consider the special fibre ring $\mathcal{F}_I(S) = \gr_I(S) \otimes_S S/\mathfrak{n}$, then it is standard graded over the field $S/\mathfrak{n}=k$. We may write $\mathcal{F}_I(S)=k[T_1,...,T_r]/J$ for some homogeneous ideal $J$ such that the image of $x_i$ is $T_i+J$ for $1 \leq i \leq r$. Let $\Gamma=\Mon(k[T_1,...,T_r]) \backslash \Mon(in(J))$, where the initial is taken with respect to any term order which is a refinement of the partial order given by the total degree. Then by Proposition 2.2.5 in \cite{eisenbud2013commutative}, the monomials in $\Gamma_i$ is a $k$-basis of $(k[T_1,...,T_r]/J)_i=(\mathcal{F}_I(S))_i=I^i/I^{i+1} \otimes_S S/\mathfrak{n}$. So taking the preimage, we know that $u(x),u \in \Gamma_i$ is an $S/I$-basis of $I^i/I^{i+1}$.

(2) implies (1): trivial.

(2) implies (3): Suppose (2) is true. Let $f \in I^i$. Since $I^i/I^{i+1}$ is generated by $u(x),u \in \Gamma_i$, $f+I^{i+1}=\sum_{u \in \Gamma_i}f_uu(x)+I^{i+1}$. So $f=\sum_{u \in \Gamma_i}f_uu(x)+g, g \in I^{i+1}$. If there is another representation $\sum_{u \in \Gamma_i}f'_uu(x)+g', g \in I^{i+1}$, then in $I^i/I^{i+1}$ we have that $\sum_{u \in \Gamma_i}f'_uu(x)=\sum_{u \in \Gamma_i}f'_uu(x)$. But $\{u(x),u \in \Gamma_i\}$ is an $S/I$-basis of $I^i/I^{i+1}$, so $f_u=f'_u$ modulo $I$. By our choice of $f_u$ and $f'_u$, $f_u,f'_u \in \sigma(S/I)$, which implies $f_u=\sigma(f_u+I)=\sigma(f'_u+I)=f'_u$. This proves (3).

(3) implies (2): Suppose (3) is true. By the existence and the uniqueness of the representation of every element in $I^i$ modulo $I^{i+1}$, we know that $I^i/I^{i+1}$ is free over $S/I$ with basis $u(x)$, with $u \in \Gamma_i$.

(3) implies (4): use \Cref{lem: expansioni1i2i3} and induct on $j-i$.

(4) implies (3): trivial.

(4) implies (5): use \Cref{lem: expansionijinfty}.

(5) implies (4): use \Cref{lem: expansioni1i2i3} for $i_3=\infty$.
\end{proof}
\begin{caution}\label{caution-expansion}
We have the following implication:
\begin{align*}
\forall i,j,I \textup{ is }\Gamma-\textup{ expandable from degree } i \textup{ to } j\\
\xRightarrow{}I \textup{ is }\Gamma-\textup{ expandable from degree } 0 \textup{ to } \infty (I \textup{ is }\Gamma\textup{-expandable).}    
\end{align*}
However, the reverse implication may not hold. If $I$ is strongly Lech-independent, then by \Cref{prop: SLI=expandable} it is $\Gamma$-expandable for some $\Gamma$, but the converse is false.
\end{caution}

For a strongly Lech-independent ideal $I$, the choice of $\Gamma$ satisfying (2)-(5) of \Cref{prop: SLI=expandable} is not necessarily unique. However, \Cref{prop: expansiondimeinv} below shows that $\dim (\Gamma)$ and $e(\Gamma)$ are independent of the choice of $\Gamma$.
\begin{proposition}\label{prop: expansiondimeinv}Let $I$ be a strongly Lech-independent ideal of a local ring $(S,\mathfrak{n})$. Then $\dim (\Gamma)$ and $e(\Gamma)$ are independent of the choice of $\Gamma$ whenever $I$ is $\Gamma$-expandable from degree $i$ to $j$ for any $i<j$. If moreover $S/I$ is Artinian, then $\dim(\Gamma)=\dim S$ and $e(I)=l(S/I)e(\Gamma)$. In particular, if $I$ is the maximal ideal $\mathfrak{n}$, then $e(\Gamma) = e(S)$.
\end{proposition}
\begin{proof}We know that
$$HS_{P/I_{\Gamma}}(z) = \sum_{i \geq 0} |\Gamma_i|z^i.$$
Since $|\Gamma_i|= \textup{rank}_{S/I}I^i/I^{i+1}$ is independent of the choice of $\Gamma$, so is $HS_{P/I_{\Gamma}}(z)$; and $\dim (\Gamma)$ and $e(\Gamma)$ only depend on $HS_{P/I_{\Gamma}}(z)$, hence they are also independent of the choice of $\Gamma$. 

Now we assume $S/I$ is Artinian. Then $\dim S = \dim \gr_I(S)$ and $\gr_I(S)$ is flat over $S/I=\gr_I(S)_0$, so
$$\dim \gr_I(S)=\dim S/I + \dim \gr_I(S)\otimes_{S/I}S/\mathfrak{n} = \dim \mathcal{F}_I(S).$$
The $i$-th component of $\mathcal{F}_I(S)$ is $I^i/I^{i+1}\otimes_{S/I}S/\mathfrak{n}$, and
$$\textup{ rank}_{S/\mathfrak{n}}(I^i/I^{i+1}\otimes_{S/I}S/\mathfrak{n})=\textup{ rank} _{S/I}I^i/I^{i+1}=|\Gamma_i|$$
because $I^i/I^{i+1}$ is free over $S/I$. This means $HS_{P/I_{\Gamma}}(z)=HS_{\mathcal{F}_I(S)}(z)$ which implies $\dim P/I_{\Gamma} = \dim \mathcal{F}_I(S) = \dim S$. Finally,
$$e(I)=\lim_{i \to \infty} (d-1)!l(I^i/I^{i+1})/i^{d-1}$$
and
$$e(P/I_{\Gamma})=\lim_{i \to \infty} (d-1)!|\Gamma_i|/i^{d-1}.$$
But $l(I^i/I^{i+1})=|\Gamma_i|l(S/I)$. So $e(I)=l(S/I)e(\Gamma)$. The last statement is obvious by taking $I=\mathfrak{n}$.
\end{proof}

Strong Lech-independence implies Lech-independence, but not conversely. For example, if $S=S_0[[x]]/\mathfrak{n}_0x^2$ and $I=(x)$. Then $I$ is Lech-independent, but not strongly Lech-independent.

One important source of strongly Lech-independent ideals is given by the following proposition:
\begin{proposition}\label{prop: SLI mS}Suppose $(R,\mathfrak{m}) \to (S,\mathfrak{n})$ is a flat local map, and $J$ is a strongly Lech-independent ideal in $R$. Pick any $\Gamma$ such that $J$ is $\Gamma$-expandable from degree $i$ to $j$ for any $i<j$. Such $\Gamma$ exists by \Cref{prop: SLI=expandable}. Then $I = JS$ is strongly Lech-independent in $S$, and $I$ is $\Gamma$-expandable from degree $i$ to $j$ for any $i<j$. In particular, if $J = \mathfrak{m}$, then $I = \mathfrak{m}S$ is strongly Lech-independent. Moreover, for any $\Gamma$ such that $\mathfrak{m}S$ is $\Gamma$-expandable from degree $i$ to $\infty$ for any $i$, we have $e(\Gamma)=e(R)$.
\end{proposition}
\begin{proof}If $(R,\mathfrak{m}) \to (S,\mathfrak{n})$ is flat local map, then there is an isomorphism $I^i/I^{i+1} \cong J^i/J^{i+1} \otimes_{R/J} S/I$. Note that freeness and a basis of a module is preserved under any base change. Let $x_1,x_2,...,x_r$ be a minimal generating set of $J$, and $y_i$ be the image of $x_i$, then $y_1,y_2,...,y_r$ is a minimal generating set of $I$ because the map is local. So, if $J$ is $\Gamma$-expandable from degree $i$ to $j$ for any $i<j$, or equivalently $J$ is $\Gamma$-expandable from degree $i$ to $i+1$ for any $i$, then $u(x), u \in \Gamma_i$ is a basis of $J^i/J^{i+1}$ over $R/J$. This means $u(y), u \in \Gamma_i$ is a basis of $I^i/I^{i+1}$ over $S/I$. Hence $I$ is $\Gamma$-expandable from degree $i$ to $i+1$ for any $i$, so $I$ is $\Gamma$-expandable from degree $i$ to $j$ for any $i<j$. If $J = \mathfrak{m}$, $R/J$ is a field, so $J^i/J^{i+1}$ is free over $R/J$ and $J$ is strongly Lech-independent. We can pick a standard set $\Gamma_0$ such that $J$ is $\Gamma_0$-expandable from degree $i$ to $j$ for any $i<j$, then $I$ is also $\Gamma_0$-expandable from degree $i$ to $j$ for any $i<j$. Then $e(\Gamma) = e(\Gamma_0) = e(R)$ by \Cref{prop: expansiondimeinv}.
\end{proof}

We stress that \Cref{prop: SLI mS} does not provide all the strongly Lech-independent ideals, as shown in the following example:
\begin{example}\label{eg: SLI not mS}Let $k$ be a field, $S = k[[t,x,y]]/(t^2,x^2-ty^2)$, $I=(x,y)$. Then $I$ is strongly Lech-independent in $S$. Let $R$ be the subring generated over $k$ by $x$, $y$. Then $R=k[[x,y]]/(x^4)$ and $S$ is not flat over $R$.
\end{example}
\begin{proof}We have $\gr_I(S) = k[t,x,y]/(t^2,x^2-ty^2)$. It is a standard graded ring with $\deg t$ = 0, $\deg x$ = $\deg y$ = 1. Let $S_0 = \gr_I(S)_0 = k[t]/t^2$, then $\gr_I(S)_1=S_0x+S_0y$, and for $i \geq 2$,
$\gr_I(S)_i=\sum_{0 \leq j \leq i} S_0x^jy^{i-j}/\sum_{2 \leq j \leq i} S_0(x^jy^{i-j}-tx^{j-2}y^{i-j+2})$. We see $\gr_I(S)_i$ is free over $S_0$ for $i \geq 1$, which implies that $I$ is strongly Lech-independent. We have  $R=k[[x,y]]/((t^2,x^2-ty^2) \cap k[[x,y]])=k[[x,y]]/(x^4)$. Now $S$ has a minimal generating set $1,t$ as an $R$-module and a nontrivial relation $x^2-ty^2=0$, so $S$ is not free over $R$. Since $S$ is module-finite over $R$ and $R$ is local, $S$ is not flat over $R$.
\end{proof}

\section{strong Lech-independence and inequalities on multiplicities of ideals}\label{se:Maintheorem}
This section is the technical core of the paper and covers the proof of the main theorems on multiplicities of ideals. The idea of the proof proceeds in three steps. 
\begin{enumerate}
\item Pick out a certain subset $A$ of $S$ such that every element in $S$ is a possibly infinite $k$-linear combination of elements in $A$. The precise statement is in \Cref{lem: product representation}.
\item For every fixed $t$, we pick out a subset $A_t$ which generates $S$ over $k$ after we mod out $\mathfrak{n}^t$. Then, $|A_t|$ gives an upper bound of $l(S/\mathfrak{n}^t)$. As $e(S)=\lim_{t \to \infty}d!l(S/\mathfrak{n}^t)/t^d$, the value of $\lim_{t \to \infty}d!|A_t|/t^d$ gives an upper bound of $e(S)$. Moreover, in the standard graded case, we prove $|A_t|=l(S/\mathfrak{n}^t)$, so $e(S)=\lim_{t \to \infty}d!|A_t|/t^d$.
\item Calculate $\lim_{t \to \infty}d!|A_t|/t^d$ using power series. A suitable choice of $A_t$ makes this computation possible.
\end{enumerate}
To complete these steps, we need to find a way to record when an element of $S$ lies in $\mathfrak{n}^t$. Therefore, we recall the concept of $\mathfrak{n}$-adic order here:
\begin{definition}[$\mathfrak{n}$-adic order]\label{def: nadicorder}
The order of $f$, denoted by $\ord(f)$, is the unique integer $t$ such that $f \in \mathfrak{n}^t \backslash \mathfrak{n}^{t+1}$ if $f \neq 0$ and is $\infty$ if $f=0$.  
\end{definition}
All these steps contribute to the two main results of this paper, namely \Cref{thm: maintheoremA} and \Cref{thm: maintheoremB}.

We still follow \Cref{assumption2.1}, and set up some assumptions unique to this section.
\begin{assumption}\label{assumption5.1}
Throughout this section, unless otherwise stated, we assume the following:
\begin{enumerate}
\item $l(S/I)=l<\infty$; in other words, $I$ is $\mathfrak{n}$-primary.
\item $S$ has a coefficient field $k$.
\item $f_1,\ldots,f_l$ is a sequence in $S$ such that they form a $k$-basis of $S/I$ modulo $I$; such sequence exists according to \Cref{eg: klinearlifting}.
\item $P=k[T_1,\ldots,T_r]$ where $r$ is the number of minimal generators according to \Cref{se:notation}, and $\Gamma$ is a standard set in $\Mon(P)$.
\item $IS_c$ is $\Gamma$-expandable as $S_c$-ideal where $S_c$ is the completion of $S$.
\end{enumerate}
We also write $t_i=\ord(x_i)$ which is a positive integer because $I \subset \mathfrak{n}$.
\end{assumption}
\Cref{assumption5.1} does not give any comparison between any $t_i$'s, but we may assume additionally $t_1\leq t_2\leq \ldots \leq t_r$ by permuting $x_i$'s.

In the proof below, we will work with elements of the form $f_iu(x), u \in \Gamma$ when $f_1,\ldots,f_l,x=(x_1,\ldots,x_r), \Gamma$ are all given. The proof is essentially counting such elements with given orders. So we introduce an extra definition called the \textit{expected order} which estimates the orders of such elements.
\begin{definition}\label{def: eord}
Let $f_1,\ldots,f_l,x_1,\ldots,x_r$ be elements in $S$. We call this sequence the \textit{degree-predicting data}. For elements of the form $f_iu(x)$, where $u$ is a monomial, we define 
\textbf{one of the expected order} of $f_iu(x)$ with respect to the degree-predicting data $(f_i,x_i)$ to be
$$\eord(f_iu(x))=\ord(f_i)+\sum_{1 \leq i \leq r}\deg_{T_i}u\cdot\ord(x_i)$$
We will omit the degree-predicting data if it is clear from the context. An element $f$ may have multiple choices of the expected order only if we can write $f=f_iu(x)=f_{i'}u'(x)$ for $i \neq i'$ or $u \neq u'$; otherwise there is only one choice of expected order and we call it \textbf{the expected order} of $f=f_iu(x)$.
\end{definition}
\begin{remark}\label{rmk: ord>eord}
For general elements $x,y \in S$, $\ord(xy) \geq \ord(x)+\ord(y)$. Thus, if $u=T_1^{a_1}\ldots T_r^{a_r}$, then $f_iu(x)=f_ix_1^{a_1}\ldots x_r^{a_r}$ and $\ord(f_iu(x)) \geq \ord(f_i)+a_1\ord(x_1)+\ldots+a_r\ord(x_r)=\eord(f_iu(x))$. That is, the expected order is bounded above by the order, and it can be seen as a prediction of the order using finitely many data $\ord(f_i),\ord(x_i)$. We only define the expected order for elements of the form $f_iu(x)$ where $u$ is a monomial.    
\end{remark}
\begin{lemma}\label{lem: product representation}Assume $S$ is complete, and $\sigma:S/I \to S$ is the $k$-linear lifting which maps $f_i+I$ to $f_i$ of which the existence is guaranteed in \Cref{eg: klinearlifting}. Then expanding $f \in S$ as a $k$-linear combination of $f_i \cdot u(x)$ gives a $k$-linear isomorphism
$$S \cong \prod_{1 \leq i \leq l, u \in \Gamma} k \cdot f_iu(x).$$
\end{lemma}
\begin{proof}Since $\sigma$ is a $k$-linear lifting, $\sigma(S/I)$ is a $k$-vector space inside $S$. Since $S$ is complete, There is a natural $k$-linear map
$$\prod_{u \in \Gamma} \sigma(S/I)\cdot u(x) \to S, (f_u\cdot u(x))_u \to \sum_{u \in \Gamma}f_uu(x).$$
The sequence $x$ is $\Gamma$-expandable, so this is a set-theoretic bijection. Therefore, it is a $k$-linear isomorphism with inverse given by the unique expansion, so $\prod_{u \in \Gamma} \sigma(S/I)\cdot u(x) \cong S$. Now $\sigma(S/I)$ is a finite dimensional $k$-vector space with basis $f_1,\ldots,f_l$, so $\sigma(S/I) \cong \prod_{1 \leq i \leq l}k\cdot f_i$. The desired isomorphism is just the combination of the above two isomorphisms.
\end{proof}
\begin{notations}\label{def: At}
Set
$$A=\{g \in S|g=f_iu(x), 1 \leq i \leq l,u \in \Gamma\}$$
Note that in $A$, $f_iu(x)=f_{i'}u'(x)$ implies $i=i'$ and $u=u'$, otherwise the equality will give two different representations of the same element, contradicting the $\Gamma$-expandable property. 

We use $f_1,\ldots,f_l,x_1,\ldots,x_r$ as a degree-predicting data, then every element in $A$ has a unique expected order. For a positive integer $t$, set
$$A_t=\{g \in A|\eord(g)<t\}$$
For $t \geq 2$, set
$$\Delta A_t=A_{t+1}\backslash A_t=\{g \in A|\eord(g)=t\}$$
and set $\Delta A_1=A_1$.
\end{notations}
Under \Cref{def: At} we see the result in \Cref{lem: product representation} becomes $S \cong \prod_{g \in A} k\cdot g$.

We give the following definition of what we mean by ``up to completion":
\begin{definition}\label{def: uptocompletion}
We say the local ring $(S,\mathfrak{n})$ is standard graded up to completion if there is a standard graded ring $(S_g,\mathfrak{n}_g)$ over the field $k=S/\mathfrak{n}$ with maximal homogeneous ideal $\mathfrak{n}_g$ such that the $\mathfrak{n}$-adic completion of $S$ is isomorphic to the $\mathfrak{n}_g$-adic completion of $S_g$. We denote this common complete local ring $(S_c,\mathfrak{n}_c)$ and view both $S$ and $S_g$ as subrings of $S_c$. In this case, for an $S$-ideal $I$, we say $I$ is homogeneous up to completion and has homogeneous generators $x_1,x_2,\ldots,x_r$ if $IS_c=I_gS_c$ for a homogeneous ideal $I_g \subset S_g$, $x_1,\ldots,x_r \in S_g$, and $x_1,\ldots,x_r$ are homogeneous generators of $I_g$.  
\end{definition}
The completion of a module of finite length is itself, thus taking completion does not affect strong Lech-independence of Artinian ideals and does not change the multiplicity. Moreover, it does not change the expansion property in finite degrees:
\begin{lemma}\label{lem: uptocompletioninv}Let 
$(S, \mathfrak{n})$ be a Noetherian local ring and $(S_c,\mathfrak{n}_c)$ is the completion of $S$. Let $I$ be an $S$-ideal and $I_c=IS_c$. Assume moreover $S/I$ is Artinian. Then $S/I=S_c/I_c$ and either $I, I_c$ are both strongly Lech-independent or none of them is strongly Lech-independent. For a fixed standard set $\Gamma$, if they are strongly Lech-independent and one of them is $\Gamma$-expandable from degree $i$ to $j$ for any $i<j$, then both of them are $\Gamma$-expandable from degree $i$ to $j$ for any $i<j$.

\end{lemma}
\begin{proof}
We have $S/I \cong S_c/I_c$ and for any $t$, $I^t/I^{t+1}\cong I_c^t/I_c^{t+1}$. Thus $I^t/I^{t+1}$ is free over $S/I$ if and only if $I_c^t/I_c^{t+1}$ is free over $S_c/I_c$. Moreover, 
$I$ is generated by $x=(x_1,\ldots,x_r)$ implies that $I_c$ is also generated by $x=(x_1,\ldots,x_r)$, and the isomorphism $I^t/I^{t+1} \cong I_c^t/I_c^{t+1}$ maps $u(x), u \in \Gamma_t$ still to $u(x), u \in \Gamma_t$, thus they are either both free bases or both not free bases.
\end{proof}
\begin{lemma}\label{lem: AtspansSnt}Using \Cref{def: At}, $S/\mathfrak{n}^t$ can be spanned by $A_t$ over $k$.
\end{lemma}
\begin{proof}
We may replace $S$ by its completion to assume $S$ is complete. By the above theorem, \Cref{assumption5.1} is preserved under completion. Every element in $S/\mathfrak{n}^t$ is of the form $f+\mathfrak{n}^t$. We expand $f$ as $f=\sum_{g \in A}c_gg$ where $c_g \in k$ by \Cref{lem: product representation}. If $g \in A$ with $\eord(g)>t$ then $\ord(g)>t$ and $g=0$ in $S/\mathfrak{n}^t$. Therefore, $f=\sum_{g \in A, \eord(g)<t}c_{g}g$ in $S/\mathfrak{n}^t$. It suffices to prove this is a finite sum. Actually, since $I \subset \mathfrak{n}$, $\ord(x_i)\geq 1$, which implies $\eord(f_iu(x))\geq \deg(u)$ where $\deg$ is the total degree in $k[T_1,\ldots,T_r]$. Thus $\eord(f_iu(x))<t$ implies $\deg(u)<t$, so there are only finitely many choices of $u$, and there are also finitely many choices of $f_i$, so there are only finitely many $g \in A$ with $\eord(g)<t$.
\end{proof}

The next proposition gives the limit $\lim_{t \to \infty}\frac{d!|A_t|}{t^d}$. 
\begin{lemma}\label{lem: Atlimitcomputation}
We adopt \Cref{assumption5.1} and \Cref{def: At}. Suppose $t_1 \leq t_2 \leq ... \leq t_r$. Denote $d=\dim S$. 
\begin{enumerate}
\item Let $c(z)=\sum_{t \geq 0} c_tz^t$, where $c_t=|\Delta A_t|$. Then
    $$c(z)=P(z)HS_{\Gamma}(z^{t_1},z^{t_2},...,z^{t_r})$$
    where $P(z) \in \mathbb{Z}[z]$ with $P(1)=l=l(S/I)$ and $c(z)$ satisfies (P1), (P2$_d$), (P3$_{d+1}$) of \Cref{lem: residueleadingterm}.
\item We have 
    $$c(z)/(1-z)=\sum_{t \geq 0}|A_t|z^t.$$
\item $$\lim_{t \to \infty}\frac{d!|A_t|}{t^d}=l(S/I)\sum_{i \in \Lambda,|S_i|=d} \frac{1}{\Pi_{T_j \in S_i}t_j}.$$
\end{enumerate}
\end{lemma}
\begin{proof}
(1) By definition
$c(z)=\sum_{t \geq 0}|\Delta A_t|z^t$. But $\Delta A_t=\{g \in A| \eord(g)=t\}$. Therefore, 
\begin{align*}
c(z)=\sum_{g \in A}z^{\eord(g)}=\sum_{1 \leq i \leq l,u \in \Gamma}z^{\eord(f_iu(x))}\\
=\sum_{1 \leq i \leq l,u \in \Gamma}z^{\ord(f_i)+\sum_{1 \leq j \leq r}  \ord(x_j)\deg_{T_j}(u)}\\
=\sum_{1 \leq i \leq l} z^{\ord(f_i)}\sum_{u \in \Gamma} z^{\sum_{1 \leq j \leq r}  \ord(x_j)\deg_{T_j}(u)}\\
=\sum_{1 \leq i \leq l} z^{\ord(f_i)}\sum_{u \in \Gamma} (z^{t_1})^{\deg_{T_1}u}(z^{t_2})^{\deg_{T_2}u}\ldots(z^{t_r})^{\deg_{T_r}u}\\
=P(z)\sum_{u \in \Gamma} u(z^{t_1},z^{t_2},\ldots,z^{t_r})\\
=P(z)HS_{\Gamma}(z^{t_1},z^{t_2},\ldots,z^{t_r}).
\end{align*}
Here $P(z)=\sum_{1 \leq i \leq l} z^{\ord(f_i)}$, so $P(1)$ is the number of $f_i$'s which is $l=l(S/I)$.

Let $(u_i,S_i)_{i \in \Lambda}$ be a Stanley decomposition of $\Gamma$. Then by \Cref{prop: Stanleyconsequence} $HS_{\Gamma}(\underline{z})=\sum_{i \in \Lambda} \frac{u_i(z)}{\Pi_{T_j \in S_i}(1-z_j)}$. So
\begin{equation}
HS_{\Gamma}(z^{t_1},z^{t_2},...,z^{t_r})=\sum_{i \in \Lambda} \frac{u_i(z^{t_1},z^{t_2},...,z^{t_r})}{\Pi_{T_j \in S_i}(1-z^{t_j})}.
\end{equation}
For each individual $i$, note that
$$\frac{u_i(z^{t_1},z^{t_2},...,z^{t_r})}{\Pi_{T_j \in S_i}(1-z^{t_j})}=\frac{u_i(z^{t_1},z^{t_2},...,z^{t_r})}{(\Pi_{T_j \in S_i}(1+z+...+z^{t_j-1}))(1-z)^{|S_i|}},$$
so the order at $z = 1$ of the $i$-th term is just $|S_i|$, and the other poles are given by $t_j$-th roots of unity; every $t_j$-th root of unity is a single pole of $1/(1+z+...+z^{t_j-1})$, so the order of the $i$-th term at every pole is at most $|S_i|$. So the order of the sum at $z = 1$ is at most $\max|S_i|=d$. For each pair $(u_i,S_i)$,
$$\lim_{z \to 1}\frac{u_i(z^{t_1},z^{t_2},...,z^{t_r})}{(\Pi_{T_j \in S_i}(1+z+...+z^{t_j-1}))(1-z)^{|S_i|}}(1-z)^d=\frac{1}{\prod_{T_j \in S_i}t_j}\delta_{|S_i|,d} \geq 0$$
and there is at least one $i$ such that this limit is nonzero. Taking the sum we get
$$\lim_{z \to 1}HS_{\Gamma}(z^{t_1},z^{t_2},...,z^{t_r})(1-t)^d>0.$$
In particular, it is nonzero, so $HS_{\Gamma}(z^{t_1},z^{t_2},...,z^{t_r})$ has a pole at $z=1$ of order exactly $d$. The orders of each term at the other poles are at most $d$, so the orders of their sum at these poles are at most $d$. Multiplying $P(z)$ does not change the order at $z=1$ as $P(1)\neq 0$ and it can only decrease the order of the other poles as $P(z) \in \mathbb{Z}[z]$. This means that $c(z)$ satisfies (P1), (P2$_d$), (P3$_{d+1}$).

(2) The $t$-th coefficient of $c(z)/(1-z)=c(z)(1+z+z^2+\ldots)$ is $c_0+c_1+\ldots+c_t=|A_1|+|A_2\backslash A_1|+\ldots+|A_t\backslash A_{t-1}|=|A_t|$, which is the $t$-th coefficient of the right side.

(3) If $c(z)$ satisfies (P1), (P2$_d$), (P3$_{d+1}$), then $c(z)/(1-z)$ satisfies (P1), (P2$_{d+1}$), (P3$_{d+1}$), so we can apply \Cref{lem: residueleadingterm} to 
$c(z)/(1-z)$ where we replace $d$ by $d+1$. From L'Hospital's rule we see for $t \in \mathbb{N}_+$, $\lim_{z \to 1}(1-z)/(1-z^t)=1/t$. By (1) and \Cref{lem: residueleadingterm},
\begin{align*}
\lim_{t \to \infty}\frac{d!|A_t|}{t^d}=\lim_{z \to 1}\frac{c(z)}{1-z}\cdot(1-z)^{d+1}=\lim_{z \to 1}c(z)(1-z)^d\\
=\lim_{z \to 1}P(z)\sum_{i \in \Lambda} \frac{u_i(z^{t_1},z^{t_2},...,z^{t_r})}{\Pi_{T_j \in S_i}(1-z^{t_j})}(1-z)^d\\
=l(S/I)\lim_{z \to 1}\sum_{i \in \Lambda} \frac{u_i(1,1,...,1)}{\Pi_{T_j \in S_i}(1-z^{t_j})}(1-z)^d\\
=l(S/I)\sum_{i \in \Lambda,|S_i|=d} \frac{u_i(1,1,...,1)}{\Pi_{T_j \in S_i}t_j} = l(S/I)\sum_{i \in \Lambda,|S_i|=d} \frac{1}{\Pi_{T_j \in S_i}t_j}.
\end{align*}    
\end{proof}
\begin{theorem}\label{thm: bound}We adopt \Cref{assumption5.1} and \Cref{def: At}. Suppose $t_1 \leq t_2 \leq ... \leq t_r$. Denote $d=\dim S$. Then there is an upper bound of the multiplicity of the maximal ideal:
$$e(\mathfrak{n}) \leq e(\Gamma)l(S/I)/t_1...t_{d-1}t_d.$$
If moreover $A_t$ is $k$-linearly independent modulo $\mathfrak{n}^t$ for any $t$, then there is also a lower bound:
$$e(\mathfrak{n}) \geq e(\Gamma)l(S/I)/t_rt_{r-1}...t_{r-d+1}.$$
\end{theorem}
\begin{proof}
We see $l(S/\mathfrak{n}^t) \leq |A_t|$ by \Cref{lem: AtspansSnt}. Thus
$$e(\mathfrak{n})=\lim_{t \to \infty}\frac{d!l(S/\mathfrak{n}^t)}{t^d}\leq\lim_{t \to \infty}\frac{d!|A_t|}{t^d}=l(S/I)\sum_{i \in \Lambda,|S_i|=d} \frac{1}{\Pi_{T_j \in S_i}t_j}.$$
Since we assume $t_1 \leq t_2 \leq \ldots \leq t_r$, $|S_i|=d$ implies that $\frac{1}{\Pi_{T_j \in S_i}t_j}\leq \frac{1}{t_1t_2\ldots t_d}$, so
$$e(\mathfrak{n}) \leq l(S/I)|i \in \Lambda:|S_i|=d|\frac{1}{t_1t_2\ldots t_d}=e(\Gamma)l(S/I)/t_1t_2\ldots t_d.$$
Here $|i \in \Lambda:|S_i|=d|=e(\Gamma)$ by \Cref{prop: standardset Hilbertseries}. On the other hand, if $A_t$ is $k$-linearly independent in $S/\mathfrak{n}^t$, then their images in $S/\mathfrak{n}^t$ form a $k$-basis of $S/\mathfrak{n}^t$, so $l(S/\mathfrak{n}^t)= |A_t|$. In this case
$$e(\mathfrak{n})=\lim_{t \to \infty}\frac{d!l(S/\mathfrak{n}^t)}{t^d}=\lim_{t \to \infty}\frac{d!|A_t|}{t^d}=l(S/I)\sum_{i \in \Lambda,|S_i|=d} \frac{1}{\Pi_{T_j \in S_i}t_j}.$$
When $|S_i|=d$, we have $\frac{1}{\Pi_{T_j \in S_i}t_j}\geq \frac{1}{t_rt_{r-1}\ldots t_{r-d+1}}$, so
$$e(\mathfrak{n}) \geq l(S/I)|i \in \Lambda:|S_i|=d|\frac{1}{t_rt_{r-1}\ldots t_{r-d+1}}=e(\Gamma)l(S/I)/t_rt_{r-1}\ldots t_{r-d+1}.$$
\end{proof}
The extra assumption on $A_t$ in the above theorem is quite strong and is false in general. However, it can be satisfied in the standard graded case. 
\begin{theorem}\label{thm: boundstandardgrd}
Let $(S, \mathfrak{n})$ be a local ring. Assume:
\begin{enumerate}
\item Up to completion $S$ is standard graded;
\item $I$ is $\mathfrak{n}$-primary;
\item Up to completion $I$ is homogeneous with homogeneous generators of degrees $t_1\leq t_2 \leq \ldots \leq t_r$.
\item $x_1,\ldots,x_r$ is $\Gamma$-expandable for some standard set $\Gamma$ in the completion of $S$.
\end{enumerate}
Then $e(S)=e(\mathfrak{n}) \geq e(\Gamma)l(S/I)/t_rt_{r-1}...t_{r-d+1}$. 
\end{theorem}
\begin{proof}We may assume $S$ is complete by taking completion. Condition (1), (2) and (4) imply that $S,I,\Gamma$ satisfy \Cref{assumption2.1} and \Cref{assumption5.1}, and we can choose $f_1,\ldots,f_l \in S$ satisfying \Cref{assumption5.1}. We assume $S$ is the completion of $S_g$ with respect to $\mathfrak{n}_g$. Moreover in $S$ we have $\ord(x_i)=t_i$. Since $I_g$ is homogeneous, we may choose a $k$-basis $f_i+I$ of $S/I=S_g/I_g$ such that each $f_i$ is homogeneous in $S_g$; here we view $S_g$ as a subring of $S$. Also the homogeneous minimal generators $x_1,...,x_r$ are in $S_g$. We now claim that the set $A_t$ is $k$-linearly independent in $S/\mathfrak{n}^t$. Assume $h=\sum_{g \in A_t}c_gg \in S$ is a sum satisfying $c_g \in k$ where $c_g$'s are not all $0$. We have $\eord(g)<t$ for any $c_g \neq 0$. Note that since all $f_i$'s and $x_i$'s are homogeneous elements in $S_g$, every $g \in A_t$ is a homogeneous element in $S_g$, thus $h \in S_g$. We have $h \neq 0$, otherwise we get two distinct expansions of the $h=0$ element; one of the expansion is $\sum_{g \in A_t}c_gg \in S$ and the other is the expansion with $c_g=0$ for all $g \in A$. This violates the unique expansion property, so $h \neq 0$. For $g \in A_t$, $\eord(g)=\ord(g)=\deg(g)<t$. So $h\neq 0$ can only have nonzero components in degree smaller than $t$, and in particular, it does not lie in $\mathfrak{n}_g^t$. Thus $h \notin \mathfrak{n}^t$, because $\mathfrak{n}^t \cap S_g = \mathfrak{n}_g^t$. So $A_t$ is $k$-linearly independent modulo $\mathfrak{n}^t$. Since this is true for any $t$, \Cref{thm: bound} implies that $e(S)=e(\mathfrak{n}) \geq e(\Gamma)l(S/I)/t_rt_{r-1}...t_{r-d+1}$. 
\end{proof}
\begin{theorem}\label{thm: maintheoremA}
Let $(R, \mathfrak{m}) \to (S, \mathfrak{n})$ be a flat local map of Noetherian local rings. Assume:
\begin{enumerate}
\item Up to completion $S$ is standard graded;
\item $\dim R=\dim S$;
\item Up to completion $\mathfrak{m}S$ is homogeneous with homogeneous generators of degrees $t_1\leq t_2 \leq \ldots \leq t_r$.
\end{enumerate}
Then $e(S) \geq e(R)t_1...t_{r-d}$.    
\end{theorem}
\begin{proof}
We take $I=\mathfrak{m}S$. We may assume $S$ is complete by taking completion; in this case condition (1) implies that $S$ has a coefficient field. Condition (2) implies $I$ is $\mathfrak{n}$-primary; $I=\mathfrak{m}S$ is strongly Lech-independent by \Cref{prop: SLI mS}, so $(S,I)$ satisfies conditions (1)-(3) of \Cref{thm: boundstandardgrd}. In particular, $I$ is Lech-independent, so by Hanes' result in \cite{hanes2001length}, $l(S/I) \geq t_1t_2...t_r$. Also $\mathfrak{m}$ is $\Gamma'$-expandable for some $\Gamma'$, and in this case $I$ is also $\Gamma'$-expandable and $\Gamma'$ satisfies condition (4) of \Cref{thm: boundstandardgrd}. We have $e(R) = e(\Gamma') = e(\Gamma)$ by \Cref{prop: expansiondimeinv}. Now by \Cref{thm: boundstandardgrd}, $e(S)=e(\mathfrak{n}) \geq e(\Gamma)l(S/I)/t_rt_{r-1}...t_{r-d+1}\geq e(R)t_1t_2\ldots t_r/t_rt_{r-1}...t_{r-d+1}=e(R)t_1...t_{r-d}$.    
\end{proof}
\begin{remark}\label{rmk: morethanHanes}\Cref{thm: maintheoremA} is a generalization of some of Hane's results, for example, Corollary 3.2 of \cite{hanes2001length}. We make no assumptions on the minimal reduction of $\mathfrak{m}$ or $\mathfrak{m}S$. For example, consider $R=k[[x,y^2]]/xy^2 \to S=k[[x,y]]/xy^2$. Then neither $x$ or $y^2$ can be a minimal reduction of $\mathfrak{m}$. The minimal reduction consists of one element which is a linear combination of $x$ and $y^2$ which is not homogeneous in $S$. So we cannot use Hane's result, but we can apply \Cref{thm: maintheoremA} to prove $e(R) \leq e(S)$. In fact, $e(R)=2<3=e(S)$.
\end{remark}
We can strengthen the inequality in \Cref{thm: bound} using the asymptotic Samuel function.
\begin{definition}\label{def: barv}The asymptotic Samuel function is $\bar{v}:S \to \mathbb{R} \cup \{\infty\}$ such that $\bar{v}(x)=\lim_{n \to \infty}\ord(x^n)/n$.
\end{definition}
\begin{proposition}\label{prop: barvrational}Let $S$ be a local ring.
\begin{enumerate}
\item $\bar{v}$ is well-defined, that is, the limit exists for any $x \in S$.

\item $\bar{v}$ has values in $\mathbb{Q} \cup \{\infty\}$.

\item $\bar{v}(x) \geq \ord(x)$.
\end{enumerate}
\end{proposition}
\begin{proof}For (1) and (2) see Chapter 6 and 10 of \cite{huneke2006integral}. (3) is true as $\ord(x^n) \geq n\cdot \ord(x)$.
\end{proof}
\begin{notations}\label{def: Eord}
We keep \Cref{assumption5.1} and assume $A=\{f_iu(x)|1 \leq i \leq l,u \in \Gamma\}$ as in \Cref{def: At}. Choose any sequence of positive rational numbers $s=(s_1<s_2<\ldots<s_r)$ and fix $N$ such that $Ns_i \in \mathbb{N}$ for all $s$. For $g=f_iu(x) \in A$, define 
$$\Eord_s(f_iu(x))=\ord(f_i)+\sum_{1 \leq i \leq r}\deg_{T_i}u\cdot s_i \in \frac{1}{N}\mathbb{N}.$$
By the unique expansion property, $\Eord_s(g)$ is uniquely defined on $A$. For a positive integer $t$, set
$$B_t=\{g \in A|N\Eord(g)<t\}.$$
For $t \geq 2$, set
$$\Delta B_t=B_{t+1}\backslash B_t=\{g \in A|N\Eord(g)=t\}$$
and set $\Delta B_1=B_1$.
\end{notations}
\begin{lemma}\label{lem: Btlimit}
We adopt \Cref{assumption5.1} and \Cref{def: Eord}.
\begin{enumerate}
\item Let $\tilde{c}(z)=\sum_{t \geq 0} \tilde{c}_tz^t$, where $\tilde{c}_t=|\Delta B_t|$. Then
    $$\tilde{c}(z)=P(z)HS_{\Gamma}(z^{Ns_1},z^{Ns_2},...,z^{Ns_r})$$
    where $P(z) \in \mathbb{Z}[z]$ with $P(1)=l=l(S/I)$ and $\tilde{c}(z)$ satisfies (P1), (P2$_d$), (P3$_{d+1}$) of \Cref{lem: residueleadingterm}.
\item We have 
    $$\tilde{c}(z)/(1-z)=\sum_{t \geq 0}|B_t|z^t.$$
\item $$\lim_{t \to \infty}\frac{d!|B_t|}{t^d}=l(S/I)\sum_{i \in \Lambda,|S_i|=d} \frac{1}{\Pi_{T_j \in S_i}(Ns_j)}=\frac{l(S/I)}{N^d}\sum_{i \in \Lambda,|S_i|=d} \frac{1}{\Pi_{T_j \in S_i}s_j}.$$
\end{enumerate}
\end{lemma}
Since $Ns_i \in \mathbb{N}$, the proof of the above theorem is exactly the same as \Cref{lem: Atlimitcomputation}, so we omit it.
\begin{lemma}\label{lem: BtspansSnt}Under \Cref{assumption5.1}, we assume $\bar{v}(x_i)> s_i$ for all $i$. Then:
\begin{enumerate}
\item For large $t$, $\Eord_s(g)\geq t$ implies $\ord(g)\geq t$, that is, $g \in \mathfrak{n}^t$.
\item For large $t$, $S/\mathfrak{n}^t$ can be spanned by $B_{Nt}$ over $k$.
\end{enumerate}

\end{lemma}
\begin{proof}
(1): We choose $\delta>0$ such that $\bar{v}(x_i)>s_i-\delta$ for all $i$. By definition of $\bar{v}(x_i)$, we can choose $n_0$ such that $\ord(x_i^n)\geq ns_i+n\delta$ for $n \geq n_0$. For a multi-index $a_1,a_2,\ldots,a_n$, write $|a|=\sum_{1 \leq i \leq n}a_i$. Assume $\Lambda_1=\{i:a_i>n_0\}$ and $\Lambda_2=\{i:a_i\leq n_0\}$, then
$$\ord(x_1^{a_1}x_2^{a_2}\ldots x_n^{a_n}) \geq \ord(\prod_{i \in \Lambda_1}x_i^{a_i})\geq \prod_{i \in \Lambda_1}\ord(x_i^{a_i})\geq \sum_{i \in \Lambda_1}(a_is_i+a_i\delta)$$
$$\overset{(\cdot)}{\geq} \sum_{1 \leq i \leq r}(a_is_i+a_i\delta)-rn_0(\sum_{1 \leq i \leq n}s_i+\delta)=\sum_{1 \leq i \leq r}a_is_i+|a|\delta-rn_0(\sum_{1 \leq i \leq n}s_i+\delta)$$
where $(\cdot)$ comes from the fact that $\sum_{i \in \Lambda_2}(a_is_i+a_i\delta) \leq \sum_{i \in \Lambda_2}(n_0s_i+n_0\delta) \leq rn_0(\sum_{1 \leq i \leq n}s_i+\delta)$. The inequality can be rewritten as
$$\ord(u(x)) \geq \Eord(u(x))+\deg(u)\delta-rn_0(\sum_{1 \leq i \leq n}s_i+\delta)$$
We set $D=\max\{\ord(f_i), 1 \leq i \leq l\}$ and $C=rn_0(\sum_{1 \leq i \leq n}s_i+\delta)+D$. Use the fact that $\Eord(f_iu(x))=\ord(f_i)+\Eord(u(x))$, we get
\begin{align*}
\ord(f_iu(x)) \geq \ord(u(x)) \geq \Eord(u(x))+\deg(u)\delta-rn_0(\sum_{1 \leq i \leq n}s_i+\delta)\\
\geq\Eord(f_iu(x))+\deg(u)\delta-D-rn_0(\sum_{1 \leq i \leq n}s_i+\delta)=\Eord(f_iu(x))+\deg(u)\delta-C.    
\end{align*}
Therefore, if $\ord(f_iu(x))\leq \Eord(f_iu(x))$, then $\deg(u)<C/\delta<\infty$. There are only finitely many choices of $u$ that satisfy this condition and there are only finitely many choices of $f_i$, thus there are finitely many $g \in A$ with $\ord(g)\geq \Eord(g)$. Since every $g \in A$ has a finite expected order, for $t \gg 0$ these choices of $g$ satisfy $\Eord_s(g)<t$. Thus for $t \gg 0$, $\Eord_s(g)\geq t$ implies $\ord(g)>\Eord_s(g)\geq t$.

(2): We may replace $S$ by its completion to assume $S$ is complete. \Cref{assumption5.1} is preserved under completion and the asymptotic Samuel function is unchanged by completion. Every element in $S/\mathfrak{n}^t$ is of the form $f+\mathfrak{n}^t$. We expand $f$ as $f=\sum_{g \in A}c_gg$ where $c_g \in k$ by the unique expansion property. If $g \notin B_{Nt}$, then $\Eord(g)>t$, $\ord(g)>t$ and $g=0$ in $S/\mathfrak{n}^t$. Therefore, $f=\sum_{g \in B_{Nt}}c_{g}g$ in $S/\mathfrak{n}^t$. It suffices to prove this is a finite sum. Actually, $\Eord(f_iu(x))\geq s_1\deg(u)$ where $\deg$ is the total degree in $k[T_1,\ldots,T_r]$. Thus $\Eord(f_iu(x))<t$ implies $\deg(u)<t/s_1$, so there are only finitely many choices of $u$, and there are also finitely many choices of $f_i$. Therefore, there are only finitely many $g \in A$ with $\Eord(g)<t$.
\end{proof}
\begin{theorem}\label{thm: maintheoremB0}Assume \Cref{assumption5.1} holds. Denote $\bar{v}(x_i)=q_i$ and assume that $q_1 \leq q_2 \leq ... \leq q_r$. Then $e(S) \leq e(\Gamma)l(S/I)/q_1...q_{d-1}q_d$ and $q_d < \infty$. If moreover $I$ is strongly Lech-independent, then $e(S) \leq e(I)/q_1...q_{d-1}q_d$.
\end{theorem}
\begin{proof}For sufficiently small rational number $\delta>0$, we take any $s_i=q_i-\delta$, then $0<s_1\leq s_2\leq\ldots\leq s_r$. \Cref{lem: BtspansSnt} implies $l(S/\mathfrak{n}^t)\leq |B_{Nt}|$ for $t \gg 0$. Therefore, by \Cref{lem: Btlimit}, we have
$$e(\mathfrak{n})=\lim_{t \to \infty}\frac{d!l(S/\mathfrak{n}^t)}{t^d}\leq\lim_{t \to \infty}\frac{d!|B_{Nt}|}{t^d}=N^d\lim_{t \to \infty}\frac{d!|B_t|}{t^d}$$
$$=l(S/I)\sum_{i \in \Lambda,|S_i|=d} \frac{1}{\Pi_{T_j \in S_i}s_j}\leq l(S/I)|i \in \Lambda,|S_i|=d| \frac{1}{s_1\ldots s_{d-1}s_d}$$
$$=e(\Gamma)l(S/I)/s_1\ldots s_{d-1}s_d.$$
Now let $\delta \to 0$, $s_i \to q_i$, so we get the desired inequality. The last equality is true because $I$ is strongly Lech-independent implies $e(I)=e(\Gamma)l(S/I)$.
\end{proof}
\begin{theorem}\label{thm: maintheoremB}
Assume $S$ has equal characteristic, $I$ is an $S$-ideal which is strongly Lech-independent and $l(S/I)<\infty$. Let $d=\dim S$. Assume $I$ is minimally generated by $(x_1,...,x_r)$, $\bar{v}(x_i)=q_i$ where $q_1 \leq q_2 \leq ... \leq q_r$. Then $e(S) \leq e(I)/q_1...q_{d-1}q_d$.    
\end{theorem}
\begin{proof}
We may assume $S$ is complete by completion, then $S$ has a coefficient field. Since $I$ is strongly Lech-independent, there is a standard set $\Gamma$ such that $I$ is $\Gamma$-expandable, thus \Cref{assumption5.1} is satisfied for the data $(S,I,x,\Gamma)$. So the result follows from \Cref{thm: maintheoremB0}.   
\end{proof}

\section*{Acknowledgements}
The author was supported in part by NSF-FRG grant DMS-1952366 and the Ross-Lynn Research Scholar Fund. The author would like to thank Linquan Ma for introducing this problem, providing references and explaining the proof of Lech's conjecture in dimension 3 and for standard graded base rings. The author would also like to thank Farrah Yhee for helpful comments.

\bibliographystyle{plain}
\bibliography{refsLILC}

\begin{thebibliography}{10}

\bibitem{permissibleref}
Bruce~Michael Bennett.
\newblock On the characteristic functions of a local ring.
\newblock {\em Annals of Mathematics}, 91(1):25--87, 1970.

\bibitem{eisenbud2013commutative}
David Eisenbud.
\newblock {\em Commutative algebra: with a view toward algebraic geometry},
  volume 150.
\newblock Springer Science \& Business Media, 2013.

\bibitem{hanes2001length}
Douglas Hanes.
\newblock Length approximations for independently generated ideals.
\newblock {\em Journal of Algebra}, 237(2):708--718, 2001.

\bibitem{hanes1999special}
Douglas~Allen Hanes.
\newblock {\em Special conditions on maximal Cohen-Macaulay modules, and
  applications to the theory of multiplicities}.
\newblock University of Michigan, 1999.

\bibitem{herzog1990lech}
Bernd Herzog.
\newblock Lech-hironaka inequalities for flat couples of local rings.
\newblock {\em manuscripta mathematica}, 68(1):351--371, 1990.

\bibitem{herzog1991linear}
Jurgen Herzog, Bernd Ulrich, and J{\"o}rgen Backelin.
\newblock Linear maximal cohen-macaulay modules over strict complete
  intersections.
\newblock {\em Journal of Pure and Applied Algebra}, 71(2-3):187--202, 1991.

\bibitem{huneke2006integral}
Craig Huneke and Irena Swanson.
\newblock {\em Integral closure of ideals, rings, and modules}, volume~13.
\newblock Cambridge University Press, 2006.

\bibitem{lech1960note}
Christer Lech.
\newblock Note on multiplicities of ideals.
\newblock {\em Arkiv f{\"o}r Matematik}, 4(1):63--86, 1960.

\bibitem{lech1964inequalities}
Christer Lech.
\newblock Inequalities related to certain couples of local rings.
\newblock {\em Acta Mathematica}, 112:69--89, 1964.

\bibitem{ma2014frobenius}
Linquan Ma.
\newblock {\em The Frobenius Endomorphism and Multiplicities.}
\newblock PhD thesis, 2014.

\bibitem{ma2017lech}
Linquan Ma.
\newblock Lech's conjecture in dimension three.
\newblock {\em Advances in Mathematics}, 322:940--970, 2017.

\bibitem{ma2020lim}
Linquan Ma.
\newblock Lim ulrich sequences and lech’s conjecture.
\newblock {\em Inventiones mathematicae}, 231(1):407--429, 2023.

\bibitem{martin1966complex}
D~Martin and LV~Ahlfors.
\newblock {\em Complex analysis}.
\newblock New York: McGraw-Hill, 1966.

\bibitem{Matsumuracommutativering}
Hideyuki Matsumura.
\newblock {\em Commutative ring theory}.
\newblock Number~8. Cambridge university press, 1989.

\bibitem{SturmfelsCA}
Ezra Miller and Bernd Sturmfels.
\newblock {\em Combinatorial commutative algebra}, volume 227.
\newblock Springer Science \& Business Media, 2005.

\bibitem{vasconcelos2004computational}
Wolmer Vasconcelos.
\newblock {\em Computational methods in commutative algebra and algebraic
  geometry}, volume~2.
\newblock Springer Science \& Business Media, 2004.

\end{thebibliography}

\end{document}